\documentclass[12pt,a4paper]{amsart}
\usepackage{mathtools,amssymb,amsmath,amsopn,amsthm,graphicx,MnSymbol,centernot,stmaryrd,array}
\usepackage{tikz-cd}
\usepackage{enumitem}
\usepackage{amsaddr}
\usepackage{etoolbox}
\patchcmd{\subsection}{-.5em}{.5em}{}{}
\usetikzlibrary{matrix}
\begin{document}

\newtheorem{definition}{Definition}[subsection]
\newtheorem{definitions}[definition]{Definitions}
\newtheorem{deflem}[definition]{Definition and Lemma}
\newtheorem{lemma}[definition]{Lemma}
\newtheorem{proposition}[definition]{Proposition}
\newtheorem{theorem}[definition]{Theorem}
\newtheorem{corollary}[definition]{Corollary}
\newtheorem{cors}[definition]{Corollaries}
\theoremstyle{remark}
\newtheorem{rmk}[definition]{Remark}
\theoremstyle{remark}
\newtheorem{remarks}[definition]{Remarks}
\theoremstyle{remark}
\newtheorem{notation}[definition]{Notation}
\theoremstyle{remark}
\newtheorem{example}[definition]{Example}
\theoremstyle{remark}
\newtheorem{examples}[definition]{Examples}
\theoremstyle{remark}
\newtheorem{dgram}[definition]{Diagram}
\theoremstyle{remark}
\newtheorem{fact}[definition]{Fact}
\theoremstyle{remark}
\newtheorem{illust}[definition]{Illustration}
\theoremstyle{remark}
\newtheorem{que}[definition]{Question}
\theoremstyle{definition}
\newtheorem{conj}[definition]{Conjecture}
\newtheorem{scho}[definition]{Scholium}
\newtheorem{por}[definition]{Porism}
\DeclarePairedDelimiter\floor{\lfloor}{\rfloor}

\renewenvironment{proof}{\noindent {\bf{Proof.}}}{\hspace*{3mm}{$\Box$}{\vspace{9pt}}}

\author[Gupta, Kuber, Sardar]{Esha Gupta, Amit Kuber \and Shantanu Sardar}
\address{Department of Mathematics and Statistics\\Indian Institute of Technology, Kanpur\\Uttar Pradesh, India}
%\affil{}
\email{eshag@iitk.ac.in, askuber@iitk.ac.in, ssardar@iitk.ac.in}
\title{{On the stable radical of some non-domestic string algebras}}

\begin{abstract}
We introduce the concept of a prime band in a string algebra $\Lambda$ and use it to associate to $\Lambda$ its finite bridge quiver. Then we introduce a new technique of `recursive systems' for showing that a graph map between finite dimensional string modules lies in its stable radical. Further we study two classes of non-domestic string algebras in terms of some connectedness properties of its bridge quiver. `Meta-$\bigcup$-cyclic' string algebras constitute the first class that is essentially characterized by the statement that each finite string is a substring of a band. Extending this class we have `meta-torsion-free' string algebras that are characterized by a dichotomy statement for ranks of graph maps between string modules--such maps either have finite rank or are in the stable radical. Their stable ranks can only take values from $\{\omega,\omega+1,\omega+2\}$.
\end{abstract}

\maketitle
%\tableofcontents

\newcommand\A{\mathcal{A}}
\newcommand\C{\mathcal{C}}
\newcommand\D{\mathcal{D}}
\newcommand\HH{\mathcal{H}}
\newcommand\K{\mathcal{K}}
\newcommand\LL{\mathcal{L}}
\newcommand\M{\mathrm{M}}
\newcommand\Q{\mathcal{Q}}
\newcommand\T{\mathcal{T}}
\newcommand\U{\mathcal{U}}
%notations for rings of numbers
\newcommand{\N}{\mathbb{N}} 
\newcommand{\R}{\mathbb{R}}
\newcommand{\Z}{\mathbb{Z}}
\newcommand{\bb}{\mathfrak b}
\newcommand{\uu}{\mathfrak u}
\newcommand{\vv}{\mathfrak v}
\newcommand{\ww}{\mathfrak w}
\newcommand{\xx}{\mathfrak x}
\newcommand{\yy}{\mathfrak y}
\newcommand{\zz}{\mathfrak z}
\newcommand{\ZZ}{\mathfrak Z}
\newcommand{\MM}{\mathfrak M}
\newcommand{\mm}{\mathfrak m}
\newcommand{\modfp}{\mathrm{mod}\mbox{-}}
\newcommand{\rad}[1]{\mathrm{rad}_{\Lambda}^{#1}}
%Notations for terms
\newcommand{\la}{l}
\newcommand{\La}{L}
\newcommand{\ra}{r}
\newcommand{\Ra}{R}
\newcommand{\lb}{\Bar{l}}
\newcommand{\Lb}{\Bar{L}}
\newcommand{\rb}{\Bar{r}}
\newcommand{\Rb}{\Bar{R}}
\newcommand{\sBa}{\sigma^{\mathrm{Ba}}}
\newcommand{\brac}[2]{\langle #1,#2\rangle}
\newcommand{\st}{\mathrm{st}}
\newcommand{\rk}{\mathrm{rk}}
%Notation for Ziegler spectrum, quiver etc.
\newcommand{\Zg}{\mathrm{Zg}(\Lambda)}
\newcommand{\Zgs}{\mathrm{Zg_{str}}(\Lambda)}
\newcommand{\STR}[1]{\mathrm{Str}(#1)}
\newcommand{\dmod}{\mbox{-}\operatorname{mod}}
\newcommand{\Ba}[1]{\mathrm{Ba}(#1)}
\newcommand{\PBa}[1]{\mathrm{PrBa}(#1)}

\section{Introduction}
One of the classical problems of representation theory is to classify all representations of a finite dimensional algebra $\Lambda$ and morphisms between them. The introduction of the Auslander-Reiten (AR) quiver which described the category $\mathrm{mod}\mbox{-}\Lambda$ of finite dimensional right $\Lambda$-modules modulo $\rad\omega$ was an important step in this direction; here $\rad{}$ stands for the radical of $\mathrm{mod}\mbox{-}\Lambda$. Butler and Ringel \cite{BR} classified the vertices of the AR quiver of a string algebra into two classes, namely the string modules and the band modules, as well as its arrows into four classes. This was taken forward by Schr\"{o}er in \cite{SchroerInfRadMod} where he computed the rank, i.e., the least ordinal $\lambda$ such that $\rad\lambda=0$, for any domestic string algebra. He showed that $\rad{\omega(n+2)}=0$, where $n$ is the maximal length of a path in its bridge quiver. He also proved that the rank of a non-domestic string algebra is not defined. This raises the question of computing the stable rank of these algebras, i.e., the smallest ordinal $\lambda$ such that $\rad\lambda = \rad{\lambda+1}$--the corresponding power of the radical is called the stable radical. The notion of a finite bridge quiver that played a crucial role in the analysis of the domestic case was unavailable in the non-domestic case.

Another well-studied Morita invariant associated to a finite dimensional algebra is its Ziegler spectrum whose points are the (isomorphism classes of) indecomposable pure-injective modules and whose topology is generated by a basis of open sets consisting of sets of pure-injective modules which do not lie in the kernel of a pp-functor. The Ziegler spectrum is quasicompact and its isolated points are precisely the vertices of the AR quiver. Ringel provided with a list of indecomposable pure-injective modules for a domestic string algebra and conjectured that it was a complete list. After various attempts the conjecture was settled in the affirmative by Prest and Puninski in \cite{PP}. Further with Laking \cite{LPP} they computed the Cantor-Bendixson (CB) rank of each point in the Ziegler spectrum of a domestic string algebra which equals its Krull-Gabriel (KG) dimension, leading to \cite[Theorem~10.19]{LPP} which states that the KG dimension of a domestic string algebra is equal to $(n+2)$, where $n$ is the length of the longest path in its bridge quiver.

In this paper we study the connections between some graph-theoretic properties of the bridge quiver of a string algebra and its representation-theoretic properties. We introduce the notion of a prime band and show that the bridge quiver constructed using only prime bands is finite. This bridge quiver coincides with the existing notion of bridge quiver for the domestic case. Furthermore, as in the domestic case, we show that the paths in this bridge quiver ``generate'' all the strings (Lemma \ref{generateallstrings}) although in the non-domestic case there could be multiple paths generating the same string. In this sense the bridge quiver ``linearizes'' the strings.

Borrowing the terminology from group theory we introduce the class of torsion-free string algebras in which, loosely speaking, we demand that every string, which can be extended to a strictly longer string, be extendable to a string indexed by $\Z$. The same requirement imposed on finite directed paths in the bridge quiver leads to the class of meta-torsion-free string algebras.

The bridge quiver of a domestic string algebra is acyclic \cite[p.44]{SchroerThesis}. At the other end of the spectrum we introduce the class of meta-$\bigcup$-cyclic string algebras characterized by strong connectedness of each connected component of their bridge quiver. In the torsion-free case such string algebras are characterized by extensibility of its strings to bands (Theorem \ref{purenondomeqdef}). The main result about these algebras (Theorem \ref{PNDmain}) states that the rank of any graph map between string modules, which does fnot lie in the stable radical, is finite.

The search for the converse to this result leads to the larger class of meta-torsion-free string algebras which are completely characterized (Theorem \ref{EPNDmain}) by the conclusion of Theorem \ref{PNDmain}. Further the analysis of graph maps involving band modules yields bounds on the stable rank of such string algebras--Corollary \ref{mtfrk} states that the stable rank of a meta-torsion-free string algebra can only take values amongst $\omega, \omega+1$ and $\omega+2$. It is also illustrated in Examples \ref{allranks} that all the three values are attained.

A new tool we introduce to show that a graph map lies in the stable radical is that of a `recursive term'--they are a compact version of factorizable systems described in \cite{MorsBetwnFinDimInfDimRep}. While a factorizable system talks about a countable collection of indecomposable modules and morphisms between them satisfying certain conditions, a recursive term condenses these conditions into a single equation. Here a `term' is a label of a graph map that is invariant under the change of a base string--this notion is introduced in \cite{SK} by the second and the third author.

The paper is organised as follows. We recall some basic definitions related to string algebras in \S\ref{stralg} whereas \S\ref{repth} is devoted to a basic introduction of string modules, band modules and graph maps between them. We recall the arithmetic of ordinal-indexed powers of the radical in \S\ref{rad}. In \S\ref{Hammock} we introduce terms with reference to hammock posets. Having covered the relevant background, we turn to define prime bands in \S\ref{primeband} and show that there are only finitely many prime bands as well as band-free strings in any string algebra. Using this finiteness property we go on to define the bridge quiver in \S\ref{primebridge}. \S\ref{roottrees} deals with showing that the bridge quiver is sufficient to talk about all the strings in the string algebra. We end the section by introducing meta-$\bigcup$-cyclic string algebras and proving some immediate results about them. We introduce recursive systems in \S\ref{recsystem} and demonstrate the existence of one in each non-domestic string algebra in \S\ref{casemetaucyclic}. The characterization of meta-torsion-free string algebras in terms of ranks of graph maps between string modules is proved in \S\ref{metatorsionfree}. We close the paper by making a conjecture about the stable rank of a special biserial algebra in \S\ref{fut}.
\subsection*{Acknowledgements} This work is supported by the \emph{Council of Scientific and Industrial Research (CSIR)} India - Research Grant No. 09/092(0951)/2016-EMR-I.

\section{Background}\label{background}
\subsection{Fundamentals of string algebras}\label{stralg}
Fix an algebraically closed field $k$. A \emph{quiver} is a directed graph possibly with parallel arrows as well as loops. We will use the quadruple $\Q = (Q_0, Q_1, s, t)$ to denote a quiver, where $Q_0$ is the set of vertices and $Q_1$ is the set of arrows, and where $s,t:Q_1\to Q_0$ denote the source and target functions. We will always use small roman letters $v,w$ to denote the vertices and $a,b,c,d$ to denote arrows of a quiver. Let us denote by $Q_1^-$ the collection, for each $b\in Q_1$, of the corresponding capital roman letter $B$. We treat $B$ as a path in $\Q$ where the arrow $b$ is traced in the reverse direction, and thus $s(B)=t(b)$ and $t(B)=s(b)$. Moreover we will use the Greek letters $\alpha,\beta,\gamma$ to denote syllables, i.e., a small or capital roman letter. For each $v\in Q_0$, we denote the lazy path (of length $0$) at $v$ by $1_{(v,1)}$ and its inverse by $1_{(v,-1)}$; clearly the source and the target of both of them is the vertex $v$.

A \emph{relation} on a quiver $\Q$ is a finite $k$-linear combination of paths with the same source and target. We use the notation $\rho$ to denote the set of relations, so that the algebra corresponding to the pair $(\Q,\rho)$ is given by $k\Q/\langle\rho\rangle$, where $\langle\rho\rangle$ is the ideal generated by $\rho$ in the path algebra $k\Q$ that is generated as a $k$-algebra by paths of non-negative length in $\Q$.

\begin{definition}
A \emph{string algebra} is an algebra $\Lambda$ presented with the help of a finite quiver $(\Q,\rho)$ where the following conditions are satisfied.
\begin{itemize}
\item[(a)] $\rho$ consists of monomials only.
\item[(b)] There are only finitely many paths in $\mathcal{Q}$ which do not have a subpath in $\langle\rho\rangle$.
\item[(c)] Any vertex of $Q_0$ is the source of at most two arrows, and the target of at most two arrows.
\item[(d)] For any arrow $b$, there is at most one arrow $c$ with $s(c)= t(b)$ and $cb \notin\rho$ and at most one arrow $a$ with $t(b)= s(a)$ and $ba \notin\rho$.
\end{itemize}
\end{definition}

\begin{examples}\label{Ex}
Here are two well-known families of string algebras.
\begin{enumerate}
    \item $\Lambda_2$ :    
    \begin{tikzcd}
    v_1 \arrow[r, "a", bend left] \arrow[r, "b"', bend  right] & v_2 \arrow[r, "c"] & v_3 \arrow[r, "d", bend left] \arrow[r, "e"', bend right] & v_4
    \end{tikzcd}
    with $\rho = \{cb, dc\}$.
    
    In fact this can be generalized to $\Lambda_n$ for all $n \geq 1$ as
    \begin{center}
    \begin{tikzcd}
    v_1 \arrow[r, "a_1", bend left] \arrow[r, "b_1"', bend right] & v_2 \arrow[r, "c_1"'] & v_3 \arrow[r, "a_2", bend left] \arrow[r, "b_2"', bend right] & v_4 \arrow[rr, dotted] &  & v_{2n-1} \arrow[r, "a_n", bend left] \arrow[r, "b_n"', bend right] & v_{2n}
    \end{tikzcd}
    \end{center}
    with $\rho = \{c_ib_i , b_{i+1}c_i \ 1 \leq i \leq n-1\}$
    \item $GP_{n,m}$ for $n, m \geq 2$:
    \begin{tikzcd}
    v \arrow["b"', loop, distance=2em, in=35, out=325] \arrow["a"', loop, distance=2em, in=215, out=145]
    \end{tikzcd}
    with $\rho = \{a^n, b^m, ab, ba\}$. This family was originally studied by Gelfand and Ponomarev.
\end{enumerate}
\end{examples}

\begin{definition}
A \emph{string} $\uu$ is either a lazy path at a vertex, its inverse or a finite word $\alpha_n\hdots\alpha_2\alpha_1$ of syllables satisfying
\begin{itemize}
\item $s(\alpha_{i+1})=t(\alpha_i)$ for $1\leq i \leq n-1$;
\item $\alpha_i$ and $\alpha_{i+1}$ are not inverses of each other for $1\leq i \leq n-1$;
\item no subpath (or its inverse) of $\uu$ is in $\rho$.
\end{itemize}
\end{definition}
We use the notation $\STR\Lambda$ to denote the collection of all finite strings of $\Lambda$ and small fractal letters to denote strings. We extend the definition of the source and target functions to all strings as $s(\uu)=s(\alpha_1)$ and $t(\uu)=t(\alpha_n)$.

Given two strings $\uu,\vv$ such that $s(\vv)=t(\uu)$, we say that the concatenation (written juxtaposition) $\vv\uu$ exists if the word $\vv\uu$ is a string.
\begin{definition}
Say that a string $\bb=\alpha_n\hdots\alpha_2\alpha_1$, $n>1$, is \emph{cyclic} if $t(\alpha_n)=s(\alpha_1)$. Say a cyclic string $\bb$ is a \emph{band} if all its powers $\bb^n$ under concatenation, $n\geq 1$, exist, $\alpha_1\in Q_1^{-1}$, $\alpha_n\in Q_1$, and if $\bb$ is not a power of any of its proper substrings, i.e., $\bb$ is primitive.
\end{definition}
We use the notation $\Ba\Lambda$ to denote the collection of all bands of $\Lambda$ up to cyclic permutation of its syllables, and the small fractal letter $\bb$ to denote a band. 

Say that a string algebra is \emph{domestic} if there are only finitely many bands otherwise say that it is \emph{non-domestic}.

\begin{examples}
    The only bands in $\Lambda_2$ from Examples \ref{Ex} are $aB, bA, dE, eD$. Since there are only finitely many bands, $\Lambda_2$ is a domestic string algebra.
    
    Now consider $GP_{2,3}$ from Examples \ref{Ex}. The strings $\bb_1 = aB$ and $\bb_2 = aB^2$ are bands. Observe that the string $\bb_1^{n_1}\bb_2^{m_1}\bb_1^{n_2}\bb_2^{m_2}\hdots \bb_1^{n_k}\bb_2^{m_k}$ is a band for all non-recurring sequences of positive integers $(n_1,m_1,n_2,m_2,\hdots ,n_k,m_k)$. Hence $GP_{2,3}$ is a non-domestic string algebra. In fact for all $n,m \geq 2,n+m\geq 5$, the algebra $GP_{n,m}$ is non-domestic.
\end{examples}

A \emph{left $\N$-string} is an infinite word $\hdots\alpha_2\alpha_1$ such that for any $n\in\N$, the word $\alpha_n\hdots\alpha_1$ is a string. A left $\N$-string in a domestic string algebra is an almost periodic string, i.e., is of the form $\prescript{\infty}{}{\bb}\uu$ for some primitive cyclic word $\bb$ and some finite string $\uu$ such that the composition $\bb\uu$ is defined, whereas such a statement is not true for non-domestic string algebras.

For technical reasons we choose and fix two arbitrary maps $\sigma,\varepsilon:Q_1\to\{-1,1\}$ satisfying the following conditions.
\begin{itemize}
\item[(a)] If $b_1\neq b_2$ are arrows with $s(b_1)= s(b_2)$, then $\sigma(b_1)= -\sigma(b_2)$.
\item[(b)] If $c_1\neq c_2$ are arrows with $t(c_1)= t(c_2)$, then $\varepsilon(c_1)= -\varepsilon(c_2)$.
\item[(c)] If $b, c$ are arrows with $s(b)= t(c)$ and $bc\notin\rho$, then $\sigma(b)= -\varepsilon(c)$.
\end{itemize}

The functions can be extended to all strings as follows: if $b$ is an arrow, set $\sigma(B):= \varepsilon(b)$, $\varepsilon(B):= \sigma(b)$; if $\uu=\alpha_n\hdots\alpha_1$ is a string of length $n\geq 1$, set $\sigma(\uu):= \sigma(\alpha_1)$ and $\varepsilon(\uu)= \varepsilon(\alpha_n)$. Finally define $\sigma(1_{(v,i)})= -i$ and $\varepsilon(1_{(v,i)})=i$. The composition $1_{(v,i)}\uu$ is defined for $\uu$ satisfying $t(\uu)=v$ provided $\varepsilon(\uu)=i$; in this case we say that $1_{(v,i)}\uu=\uu$. Similarly the composition $\uu1_{(v,i)}$ is defined for $\uu$ satisfying $s(\uu)=v$ provided $\sigma(\uu)=-i$; in this case we say that $\uu1_{(v,i)}=\uu$.

Henceforth the term `string algebra' will always mean a quiver with relations together with a choice of $\sigma$ and $\varepsilon$ maps. A string algebra is an artin algebra.

\begin{example}\label{Ex2}
A possible choice of $\sigma$ and $\varepsilon$ maps for $GP_{n,m}$ from Examples \ref{Ex} is $$\sigma(a)=\varepsilon(b)=1, \ \sigma(b)=\varepsilon(a)=-1.$$
\end{example}

\subsection{Classification of representations of string algebras}\label{repth}
Let $\Lambda$ be a string algebra. We can associate a `string module' $\M(\vv)$ to each string $\vv\in \STR\Lambda$ and a `band module' $\mathrm B(\bb,n,\lambda)$ to each triplet $(\bb,n,\lambda)$ where $\bb$ is a band, $n\in\mathbb Z^+$ and $\lambda\in k^*$. The details of the construction are omitted here and can be found in \cite[\S~2.3]{LakingThesis}. For distinct strings $\vv,\vv'$ we have $\M(\vv)\cong \M(\vv')$ if and only if $\vv'=\vv^{-1}$. Similarly $\mathrm B(\bb,n,\lambda) \cong \mathrm B(\bb',n',\lambda')$ if and only if $\bb'$ is a cyclic permutation of $\bb$ or $\bb^{-1}$, $n=n'$ and $\lambda=\lambda'$. Furthermore no string module is isomorphic to a band module.

In fact (the isomorphism classes of) the modules listed in the above paragraph are the only indecomposable finite dimensional $\Lambda$-modules, a result essentially due to \cite{GP}. This is a crucial step towards classification of finite dimensional $\Lambda$-modules. Let $\modfp \Lambda$ denote the category of finite dimensional right $\Lambda$-modules and module homomorphisms between them. A well-studied invariant associated to an artin algebra is its Auslander-Reiten (AR) quiver. It has  as its vertices the isomorphism classes of finite dimensional indecomposable $\Lambda$-modules and as its arrows the irreducible morphisms between such modules. 

Schr\"{o}er described a basis for the $k$-vector space $\mathrm{Hom}_\Lambda(M,N)$, where $M,N$ are vertices in the AR Quiver. The elements of these bases are called \emph{graph maps} between the corresponding modules. Say that a graph map is of type SB if its source is a string module and its target is a band module. Similarly we define when a graph map is of type BS, SS and BB. Below we recall a few details and more can be found in \cite[\S~2.4]{LakingThesis} and \cite{SchroerThesis}.

Let $\uu = \alpha_n \hdots \alpha_2\alpha_1 \in \STR \Lambda$. A substring $\vv = \alpha_i \hdots \alpha_j (1 \leq j \leq i \leq n)$ of $\uu$ is called an \emph{image substring} of $\uu$ if either $\alpha_{i+1}$ is inverse or $i = n$, and also either $\alpha_{j-1}$ is direct or $j = 1$. Dually a \emph{factor substring} of $\uu$ is defined.

\begin{example}
Continuing with Examples \ref{Ex} consider the string $DecaB$ in $\Lambda_2$. Then $ec$ is an image substring of $DecaB$ and $ca$ is a factor substring of $DecaB$.
\end{example}

If $\vv$ is an image substring of $\uu$, then the canonical map $\M(\vv)\to \M(\uu)$ is a monomorphism. Dually if $\vv$ is a factor substring of $\ww$, then the canonical map $\M(\ww)\to \M(\vv)$ is an epimorphism. A graph map $f : \M(\ww)\to \M(\uu)$ is a composition of such two morphisms and $\vv$ will be referred to as the string associated to $f$.

Let $\vv$ be a finite factor substring of $^\infty\bb^\infty$ for a band $\bb$. It is possible to define a set of canonical maps from $\mathrm B(\bb,n,\lambda)\to\M(\vv)$ indexed by $\mathrm{Hom}_k(k^n,k)$. A graph map $\mathrm B(\bb,n,\lambda) \to \M(\uu)$ is simply the composition of a canonical map to $\M(\vv)$ with a graph map $\M(\vv)\to\M(\uu)$, and $\vv$ will be referred to as the string associated to such a graph map. Dually we can describe graph maps from string modules to band modules.

For bands $\bb,\bb'$, let $\vv$ be a finite string that is a factor substring of $^\infty\bb^\infty$ and an image substring of $^\infty\bb'^\infty$. It is again possible to define a set of canonical maps from $\mathrm B(\bb,n,\lambda)\to\mathrm B(\bb',n',\lambda')$ indexed by $\mathrm{Hom}_k(k^n,k^{n'})$. Such canonical maps are graph maps between these band modules, and we refer to $\vv$ as the string associated to such a graph map. In case $\bb=\bb'$ and $\lambda=\lambda'$, then there are more graph maps induced by the basis elements of the $\mathrm{Hom}$-set between $(k^n,\lambda)$ and $ (k^{n'},\lambda)$ thought of as $k[T,T^{-1}]$-modules.

\subsection{Stable rank of a module category}\label{rad}
Recall that a two-sided \emph{ideal} $I$ of a category $\mathcal{C}$ is a class of morphisms of $\mathcal{C}$ such that $\alpha_1\phi, \phi\alpha_2 \in I$, whenever $\phi \in I$, $\alpha_1,\alpha_2$ are morphisms of $\mathcal{C}$ and the compositions are defined.
\begin{definition}
Let $\Lambda$ be an artin $k$-algebra. The radical of the category $\modfp\Lambda$, denoted $\rad{}$, is a two-sided ideal of $\modfp\Lambda$ generated by the non-invertible morphisms between indecomposable finite dimensional $\Lambda$-modules.
\end{definition}
To understand the structure of transfinite compositions of morphisms in the radical we define its powers indexed by ordinals. 
\begin{definition}
First set $\rad0 := \modfp\Lambda$. For $n\in \N\setminus\{0\}$, $\rad n$ denotes the ideal generated by all compositions of $n$ morphisms in $\rad{}$. For a limit ordinal $\nu$, define $$\rad\nu := \bigcap\{ \rad\lambda \mid \lambda < \nu \}.$$
If $\nu = \lambda + n$, where $n \in \N\setminus\{0\} $ and $\lambda$ is a limit ordinal, then define $$\rad\nu := (\rad\lambda)^{n+1}.$$ Finally set $$\rad\infty := \bigcap\{\rad\nu\mid \nu \mbox{ is an ordinal}\}.$$
\end{definition}
It is clear from the above definition that $$\rad0 \supseteq \rad1 \supseteq \rad2 \hdots\rad\omega\supseteq\rad{\omega+1}\hdots \supseteq \rad\infty.$$ This sequence has to eventually stabilize which motivates the following definition.
\begin{definition}
The \emph{stable rank} of $\modfp\Lambda$, denoted $\st(\Lambda)$, is the smallest ordinal $\nu$ such that $\rad\nu =\rad{\nu+1}$, and the \emph{stable radical} is the ideal $\rad{\st(\Lambda)}$.
\end{definition}
Schr\"{o}er \cite[Theorem~2,\,Theorem~3]{SchroerInfRadMod} showed that a string algebra is domestic if and only if $\rad{\st(\Lambda)}=0$. In fact, he also shows that \cite[Theorem~3]{SchroerInfRadMod} $\st(\Lambda)<w^2$ for a domestic string algebra. 

Transfinite powers of the radical allow us to partition the set of morphisms in $\modfp\Lambda$.
\begin{definition}
Given a morphism $f$ in $\modfp\Lambda$, set $\rk(f) := \nu$ if $f \in \rad\nu \setminus \rad{\nu + 1}$ and $\rk(f) := \infty$ if $f \in \rad\infty$.
\end{definition}

The following lemma helps us to do arithmetic of ranks and will be a key step for finding the rank of a `recursive term' in \S\ref{recsystem}.
\begin{lemma}\label{rkarithmetic}\cite[Lemma~0.7]{MorsBetwnFinDimInfDimRep}
Suppose that $f=hg$ are morphisms in $\modfp\Lambda$. If $rk(g),rk(h)\geq\nu \geq 1$ then $rk(f)>\nu$.
\end{lemma}

For convenience we also set $\rad\nu(M,N):=\rad\nu\cap\mathrm{Hom}_\Lambda(M,N)$ for $M,N\in\modfp\Lambda$. The invertible morphisms in $\modfp\Lambda$ are precisely the morphisms of rank $0$ while the arrows of the AR quiver of $\Lambda$ have rank $1$.

\begin{rmk}\label{sbbbrank}
Since $\rad\nu$ is an ideal for each $\nu$, it is enough to determine the ranks of the basis elements of the $\mathrm{Hom}$-sets, i.e., of the graph maps.
\begin{itemize}
    \item Let $f:\mathrm B(\bb,n,\lambda)\to\M(\uu)$ be a graph map of type BS with associated string $\vv$. Then for appropriate cyclic permutations $\bb',\bb''$ of the band $\bb$, the canonical map $\mathrm B(\bb,n,\lambda)\to\M(\vv)$ factors through $\M(\bb'\vv\bb'')$ and hence $f\in\rad\omega$. A dual argument shows that a graph map of type SB also lies in $\rad\omega$.
    \item  As a consequence if a map between two string modules factors as a composition of a map of type SB followed by a map of type BS then it lies in $\rad {\omega+1}$.
    \item Let $f:\mathrm B(\bb,n,\lambda)\to\mathrm B(\bb',n',\lambda')$ be a graph map of type BB. If it is induced by a basis element of the $\mathrm{Hom}$-set between $k[T,T^{-1}]$-modules then it is of finite rank. On the other hand if there is a finite string $\vv$ associated to $f$, then it factors through a direct sum of $n$ (or $n'$) copies of $\M(\vv)$. Therefore being a composition of a linear combination of maps of type BS followed by a linear combination of maps of type SB, we see that $f\in\rad{\omega+1}$.
 \end{itemize}
\end{rmk}

\subsection{Hammock posets for strings}\label{Hammock}
Let $v\in Q_0$. The strings with an endpoint at $v$ are partitioned into two sets $H_r(v)$ and $H_l(v)$ using the $\varepsilon$ and $\sigma$ functions. 
$$H_r(v):=\{\vv\in\STR\Lambda\mid t(\vv) = v\mbox{ and }\varepsilon(\vv)=1\},$$
$$H_l(v):=\{\vv\in\STR\Lambda\mid s(\vv) = v\mbox{ and }\sigma(\vv)=-1\}.$$
The strings in $H_l(v)$ are those extending to the left of $v$, i.e., starting at $v$ while the strings in $H_r(v)$ are those extending to the right of $v$ i.e., ending at $v$.

Consider an ordering $<_r$ on $H_r(v)$ defined by $\uu<_r\vv$ if one of the following holds:
\begin{enumerate}
\item $\vv=\uu a\xx$ for $a\in Q_1$ and $\xx\in\STR\Lambda$.
\item $\uu =\vv B\yy$ for $b\in Q_1$ and $\yy\in\STR\Lambda$.
\item $\vv=\zz a\xx$ and $\uu=\zz B\yy$ for $a,b\in Q_1$ and $\xx,\yy,\zz\in\STR\Lambda$.
\end{enumerate}
An ordering $<_l$ on $H_l(v)$ is defined by $\uu<_l\vv$ if $\vv^{-1}<_r\uu^{-1}$ in $H_r(v)$. 
Let $\mm_l$ be the direct string of maximal length and $\MM_l$ be the inverse string of maximal length in $H_l(v)$; similarly let $\MM_r$ be the inverse string of maximal length and $\mm_r$ be the direct string of maximal length in $H_r(v)$.
\begin{lemma}\cite[\S2.5]{SchroerThesis}
The posets $(H_l(v),<_l)$ and $(H_r(v),<_r)$ are total orders. The element $\mm_l$ is the minimal and $\MM_l$ the maximal element of $H_l(v)$ while $\MM_r$ is the minimal and $\mm_r$ the maximal element of $H_r(v)$. Each element of $H_{l/r}(v)$ (except the minimal element) has a direct predecessor and each element (except the maximal element) has a direct successor.
\end{lemma}

Suppose $\uu,\vv\in H_r(v)$ and $\vv$ is the direct successor of $\uu$ with respect to $<_r$. From the proof of the above lemma we know the description of these elements in terms of each other, and thus that $|\uu|\neq|\vv|$, where $|\uu|$ denotes the length of the string $\uu$. If $|\uu|<|\vv|$, then we say that $\ra(\uu):=\vv$ where $\vv$ can be written as $\uu a\xx$ for unique $a\in Q_1,\xx\in\STR\Lambda$ such that $\xx$ is the inverse string of maximal length such that $\uu a\xx$ is also a string. If $|\uu|>|\vv|$, then we say that $\rb(\vv):=\uu$ where $\uu$ can be written as $\vv B\yy$ for unique $b\in Q_1,\yy\in\STR\Lambda$ such that $\yy$ is the direct string of maximal length such that $\vv B\yy$ is also a string.

The notations $\la,\lb$ can be defined in $H_l(v)$ analogously as follows. Suppose $\uu<_l\vv$. If $|\uu|<|\vv|$, then we say that $\la(\uu):=\vv=\yy B\uu$ for unique $b\in Q_1,\yy\in\STR\Lambda$ such that $\yy$ is the direct string of maximal length such that $\yy B\vv$ is also a string, otherwise $\lb(\vv):=\uu=\xx a\uu$ for unique $a\in Q_1,\xx\in\STR\Lambda$ such that $\xx$ is the inverse string of maximal length such that $\xx a\uu$ is also a string.
\begin{example}
In $\Lambda_2$ from Examples \ref{Ex} we have $\la(a)=ecaBa$ and $\lb(a)=ca$. Note that $\la(c)$ is not defined.
\end{example}

The following technical condition on string algebras, which roughly states that whenever a string can be extended on either side, then it can be extended to an infinite string, will be useful in \S\ref{MetaUcyclic}.
\begin{definition}
We say that a string algebra $\Lambda$ is \emph{$\la$-torsion-free} (resp. \emph{$\ra$-torsion-free}) if for each $v\in Q_0$, whenever $I$ is a finite interval in the total order $H_l(v)$ (resp. $H_r(v)$) such that the sequence of lengths of strings in $I$ is monotone, there exists a string $\xx\notin I$ such that $I\cup\{\xx\}$ is also an interval in $H_l(v)$ (resp. $H_r(v)$) such that its length sequence is also monotone.

We say that a string algebra is \emph{torsion-free} if it is both $\la$ and $\ra$-torsion-free. 
\end{definition}

\begin{examples}
All the string algebras discussed in Examples \ref{Ex} are torsion-free. The following quiver with $\rho = \{cb\}$ is not a torsion-free string algebra for the reason that $\lb(a)=c$ exists while $\lb^2(a)$ does not.
\begin{center}
\begin{tikzcd}
v_1 \arrow[r, "a", bend left] \arrow[r, "b"', bend right] & v_2 \arrow[r, "c"] & v_3
\end{tikzcd}
\end{center}
\end{examples}

Let $\Lambda$ be a string algebra and let $\vv\in H_l(v)$ such that $\la(\vv)$ and $\la^{n+1}(\vv)=\la(\la^n(\vv))$ exist for all $n\in\N$. The limit of the increasing sequence $\brac 1 \la(\vv):=\lim_{n\to\infty}\la^n(\vv)$ is a left $\N$-string. Note that $\brac 1\la(\vv)\notin H_l(v)$.

The simplest question we ask is whether the equation \begin{equation}\label{lalbeqn}
    \brac 1\lb(\xx)=\brac 1\la(\vv)
\end{equation}
has a solution for a given string $\vv$. If one solution exists then infinitely many solutions exist for if $\yy'$ is a solution then so is $\lb(\yy')$. We say that $\yy$ is a \emph{fundamental solution} of Equation \eqref{lalbeqn} if $\yy=\lb(\xx)$ has no solution, i.e., if $\yy$ is a solution of minimal length. Further suppose that $\vv=\la(\xx)$ has no solution. Since both $\vv$ and $\yy$ are substrings of the left $\N$-string $\brac 1\la(\vv)=\brac 1\lb(\yy)$, one of them is a proper substring of the other. If $\vv$ is a substring of $\yy$, then we rewrite Equation \eqref{lalbeqn} as $\yy=\brac\lb\la(\vv)$, otherwise we write $\vv=\brac\la\lb(\yy)$. If $\yy=\brac\lb\la(\vv)$, then we also say that $\yy=\brac\lb\la(\la^n(\vv))$ for any $n\in\N$.

The new expression $\brac\lb\la$ labels a path in the hammock poset of a string algebra. Say a \emph{real term} is an expression formed out of $\la,\lb$, brackets and concatenation that labels a path between two strings for a domestic string algebra--the details of the construction of real terms can be found in \cite{SK}. In this paper we will not concern ourselves with real terms in view of Theorem \ref{norealterm} which loosely states that no real term, other than a finite composition of finite power of $\la,\lb,\ra,\rb$, can label a path between two strings in the class of non-domestic string algebras that are the main objects of study here.

\begin{rmk}
Note that because of the description of the strings $\la(\vv)$ and $\lb(\vv)$ for $\vv \in \STR\Lambda$, the strings $\lb(\la(\vv))$ and $\la(\lb(\vv))$ are not defined. 
\end{rmk}

\section{The bridge quiver of a string algebra}\label{BridgeQuiver}

\subsection{Prime bands}\label{primeband}
The finiteness of the set of bands is crucial in the study of domestic string algebras. On the other hand the basic obstruction for studying non-domestic string algebras is the infinitude as well as the complexity of their bands. In order to deal with this situation we introduce the concept of ``prime bands'' which play the same role in the analysis of non-domestic string algebras that is played by the collection of (finitely many) bands in the domestic case.

\begin{definition}
A band $\bb$ is a \emph{prime band} if for none of its cyclic permutation $\bb$ can be written as $\bb_1\bb_2\hdots\bb_k$
for some $k>1$ where each $\bb_i$ is a cyclic permutation of a band.
\end{definition}

Note that in a domestic string algebra all bands are prime.

\begin{example}
In $GP_{2,3}$ from Examples \ref{Ex}, we have $aB^2$ is a prime band while $aB^2aB$ is not. 
\end{example}

The main goal of this section is to show that there are only finitely many prime bands in any string algebra.

Say that a string is \emph{mixed} if it contains both direct and inverse syllables. The definition of a string algebra clearly implies that each band is a mixed primitive cyclic string.
\begin{rmk}
Suppose $\uu$ is a mixed cyclic string. Then the string $\uu^n$ is defined for all $n\geq2$ if and only if $\uu^2$ is defined. The only non-trivial point is that a substring of $\uu^n$ which is a blocked relation, i.e., an element of $\rho$, is actually a substring of $\uu^2$.
\end{rmk}

Say that a cyclic string $\uu=\alpha_n\hdots\alpha_1$ is \emph{permutable} if the words $\alpha_k\hdots\alpha_1\alpha_n\hdots\alpha_{k+1}$ are strings for each $1\leq k<n$. Now we can give a characterisation of bands.
\begin{rmk}\label{banddef}
As an easy consequence of the above remark note that a mixed primitive cyclic string is a cyclic permutation of a band if and only if it is permutable.
\end{rmk}

Now we observe how many times a mixed string can occur as a substring in a prime band.

\begin{proposition}\label{joint}
Let $\bb$ be a prime band and $aB$ be a string for $a,b\in Q_0$. Then any permutation of $\bb$ contains the substring $aB$ at most once.
\end{proposition}

\begin{proof}
Suppose, for contradiction, that the string $aB$ occurs in a permutation of $\bb$ at least twice, so that a permutation of $\bb$ is of the form $$B\uu_1aB\uu_2aB\hdots\uu_na,$$ where $n>1$ and no $\uu_i$ contains $aB$ as a substring.

We claim that each $\bb_i:=B\uu_ia$ is a cyclic permutation of a band. Clearly $\bb_i$ is a mixed cyclic string. It is also primitive since any of its permutations cannot contain more than one copy of $aB$. Therefore, in view of the above remark, it only remains to show that $\bb_i$ is permutable. A substring of a permutation of $\bb_i$ that lies in $\rho$ is also a substring of $\bb_i$ for its first and the final syllables are not both direct or both inverse. This completes the proof.
\hfill\end{proof}

Now we are ready to prove the the main result of this section.
\begin{theorem}\label{finprimeband}
There are only finitely many prime bands in a string algebra.
\end{theorem}

\begin{proof}
Let $\bb=\uu_1\vv_1\uu_2\vv_2\hdots\uu_n\vv_n$ be a prime band where each $\uu_i$ is a direct string while each $\vv_i$ is an inverse string. By Proposition \ref{joint} we know that $n$ is bounded as a function of the number of strings of the form $aB$, which clearly is finite given the finiteness of $Q_1$. Moreover the length of each direct (as well as inverse) string is absolutely bounded by the second clause in the definition of string algebras. Thus the length of a prime band is absolutely bounded which proves the statement.
\hfill\end{proof}

We can also use the idea of counting substrings of the form $aB$ to prove a result which would be useful in concluding the finiteness of the `bridge quiver' in the next section. Say that a string is \emph{band-free} if it does not contain any cyclic permutation of a band as a substring.

\begin{proposition}\label{finbandfree}
There are only finitely many band-free strings in a string algebra.
\end{proposition}

\begin{proof}
Suppose, on contrary, that there are infinitely many such strings. This implies that there are band-free strings of arbitrary lengths. Suppose in a string algebra the absolute bound on the length of a direct string is $m$ and the number of strings of type $aB$, $a,b \in Q_1$, is $n$. Then, any string of length greater than $(n+1)m + 2n$ must contain a string of type $aB$ at least twice and hence would be of the form $$\uu_1aB\uu_2aB\uu_3,$$
for strings $\uu_1, \uu_2, \uu_3$ and $a,b \in Q_1$ such that $aB$ is not a substring of $\uu_2$. Then, $B\uu_2a$ is a mixed primitive cycle. It is also permutable because a substring of any of its permutation that lies in $\rho$ would also be its substring. Therefore, by Remark \ref{banddef}, $B\uu_2a$ is a band, contrary to our assumption.
\hfill\end{proof}

This result can also be deduced from \cite[Lemma~5.2]{STV}.

\subsection{Bridges and half bridges}\label{primebridge}
In order to understand the flow in the hammock posets $H_{l/r}(v)$ for fixed $v\in Q_0$, we need to understand how to move between bands. In the domestic case this movement can be captured using some special strings called `bridges', and then we can construct a quiver whose vertices are bands and whose arrows are such bridges. In the non-domestic case we replace the bands with prime bands and define the bridges in a similar way.

Recall that $\Ba\Lambda$ is the collection of bands up to cyclic permutation. Let $Q_0^{\mathrm{Ba}}$ be a fixed set of representatives of prime bands in $\Ba\Lambda$.

\begin{definition}
For $\bb_1,\bb_2\in Q_0^{\mathrm{Ba}}$ say that a finite string $\uu$ is a \emph{weak bridge} $\bb_1\to\bb_2$ if it is band-free and if the word $\bb_2\uu\bb_1$ is a string.

Say that a weak bridge $\bb_1\xrightarrow{\uu}\bb_2$ is a \emph{bridge} if there is no prime band $\bb$ and weak bridges $\bb_1\xrightarrow{\uu_1}\bb$ and $\bb\xrightarrow{\uu_2}\bb_2$ such that one of the following holds.
\begin{itemize}
    \item $\uu=\uu_2\uu_1,|\uu_1|>0,|\uu_2|>0$.
    \item $\uu=\uu'_2\uu'_1,|\uu'_1|>0,|\uu'_2|>0,\uu_2=\uu'_2\uu''_2,\uu_1=\uu''_1\uu'_1$ and $\bb=\uu''_2\uu''_1$.
\end{itemize}
\end{definition}

While in the domestic case the above conditions imply that a bridge has non-zero length, a bridge between two prime bands in a non-domestic algebra could have zero length. In particular, we also allow zero length bridges from a band to itself. These are called \emph{trivial} bridges.

By $Q_1^{\mathrm{Ba}}$ we denote the set of all bridges between prime bands in $Q_0^{\mathrm{Ba}}$; this together with $Q_0^{\mathrm{Ba}}$ constitutes a quiver $\Q^{\mathrm{Ba}}=(Q_0^{\mathrm{Ba}},Q_1^{\mathrm{Ba}})$ known as the \emph{bridge quiver}. Theorem \ref{finprimeband} and Proposition \ref{finbandfree} together imply that the the bridge quiver of any string algebra is finite.

\begin{definition}
Given $\xx\in\STR\Lambda$ and $\bb\in Q^{\mathrm{Ba}}_0$, a \emph{weak half bridge} $\xx\to\bb$ is a string $\uu$ such that
\begin{enumerate}
    \item the word $\bb\uu\xx$ is a string (in case $\xx$ is of the form $1_{(v,i)}$, then we require that $\bb\uu$ and $\uu1_{(v,i)}$ are defined);
    \item the word $\uu$ is band-free.
\end{enumerate}
\end{definition}

Suppose $\bb_1\xrightarrow{\uu}\bb_2$ is a bridge. The \emph{exit} of this bridge is first syllable in $\prescript{\infty}{}{\bb_2}\uu\bb_1$ from the right where the strings $\prescript{\infty}{}{\bb_2}\uu\bb_1$ and $\prescript{\infty}{}{\bb_1}$ differ. We further define a function $\sBa:Q_1^{\mathrm{Ba}}\to\{1,-1\}$ by $\sBa(\uu):=1$ if and only if the exit of $\uu$ is an inverse arrow. Similarly, the \emph{exit} of a weak half bridge $\xx\xrightarrow{\uu}\bb$ is the first syllable of $\bb\uu$. We extend the map $\sBa$ to weak half bridges too.

Proposition \ref{finbandfree} says that the set of weak half bridges from $\xx$ is finite. We define a relation $\sqsubseteq$ on the set of all weak half bridges from $\xx$ by $\uu\sqsubseteq\uu'$ for distinct $\xx\xrightarrow{\uu}\bb,\ \xx\xrightarrow{\uu'}\bb'$ if there is a bridge $\bb\xrightarrow{\vv}\bb'$ such that either $\uu'=\vv\uu$ or $\uu'$ can be obtained from $\vv\uu$ by removing, if exists, a cyclic permutation of $\bb$. 

The dual notions are defined as follows. Say $\bb\xrightarrow{\uu}1_{(v,i)}$ is a \emph{reverse weak half/half bridge} if $1_{(v,-i)}\xrightarrow{\uu^{-1}}\bb^{-1}$ is a weak half/half bridge.

\begin{definition}
Given $\xx\in\STR\Lambda$ and $\bb\in Q^{\mathrm{Ba}}_0$, a weak half bridge $\xx\xrightarrow{\uu}\bb$ is a \emph{half bridge} if there does not exist a weak half bridge $\xx\xrightarrow{\uu'}\bb'$ such that $\uu'\sqsubseteq\uu,\ \uu\nsqsubseteq\uu'$ and $\sBa(\uu) = \sBa(\uu')$.
\end{definition}

\begin{rmk}
Recall that every band in a domestic string algebra is a prime band. Moreover the notions of a bridge
and (reverse) half bridge coincide with the existing notions.
\end{rmk}

\begin{definition}
Given $v_1, v_2 \in Q_0$ and $i_1, i_2 \in \{-1,1\} $, a \emph{weak zero bridge} $1_{(v_1,i_1)}\xrightarrow{\uu}1_{(v_2,i_2)}$ is a band-free string $\uu$ such that $1_{(v_2,i_2)}\uu$ and $\uu1_{(v_1,i_1)}$ are defined.

A \emph{zero bridge} $1_{(v_1,i_1)}\xrightarrow{\uu}1_{(v_2,i_2)}$ is a weak zero bridge for which there is no band $\bb$, weak half bridge $1_{(v_1,i_1)}\xrightarrow{\uu_1}\bb$ and reverse weak half bridge $\bb\xrightarrow{\uu_2}1_{(v_2,i_2)}$ such that $|\uu_1|,|\uu_2|>0$, and either $\uu=\uu_2\uu_1$ or $\uu$ can be obtained from $\uu_2\uu_1$ by removing, if exists, a cyclic permutation of $\bb$.
\end{definition}

Finally we introduce two notions which will be crucial in the next subsection. Define the \emph{extended bridge quiver} (resp. \emph{extended weak bridge quiver}), denoted $\overline{\Q}^{\mathrm{Ba}}$ (resp. $\widetilde{\Q}^{\mathrm{Ba}}$), as follows. Its vertex set is $\overline{Q}_0^{\mathrm{Ba}}=\widetilde{\Q}^{\mathrm{Ba}}_0:=Q_0^{\mathrm{Ba}}\sqcup\{1_{(v,i)}\mid v\in Q_0, i\in\{1,-1\}\}$, whereas its arrows are all the (resp. weak) bridges between prime bands, (resp. weak) half bridges from and reverse (resp. weak) half bridges to vertices of the form $1_{(v,i)}$, and (resp. weak) zero bridges between two vertices of the form $1_{(v,i)}$.
\begin{example}
Continuing with Example \ref{Ex2}, Figure \ref{extbrquiverex} shows a part of the extended bridge quiver for $GP_{2,3}$.
\begin{figure}[h]
    \begin{tikzcd}
    aB \arrow["B", loop, distance=2em, in=145, out=215] \arrow[dd, "{1_{(v,-1)}}"', bend right=49] &  &                  &    &  & bA \arrow["b"', loop, distance=2em, in=35, out=325] \arrow[dd, "{1_{(v,1)}}"'] \arrow[dd, "b", shift left] \\
    &  & {1_{(v,-1)}} \arrow[llu, "{1_{(v,-1)}}"] \arrow[llu, "a"', shift right] \arrow[llu, "B"', dashed, bend right] \arrow[lld, "{1_{(v,-1)}}"'] \arrow[lld, "a", shift left] & {1_{(v,1)}} \arrow[rru, "b"] \arrow[rru, "{1_{(v,1)}}"', shift right] \arrow[rru, "b^2", dashed, bend left] \arrow[rrd, "b"'] \arrow[rrd, "{1_{(v,1)}}", shift left] \arrow[rrd, "b^2"', dashed, bend right] &  &    \\ aB^2 \arrow[uu, "B"] \arrow[uu, "{1_{(v,-1)}}"', shift right]    &  &          &   &  & b^2A \arrow[uu, "{1_{(v,1)}}"', bend right=49] 
    \end{tikzcd}
    \caption{The bridge quiver for $GP_{2,3}$ with weak half bridges. The dotted arrows represent weak half bridges which are not half bridges.}
    \label{extbrquiverex}
\end{figure}
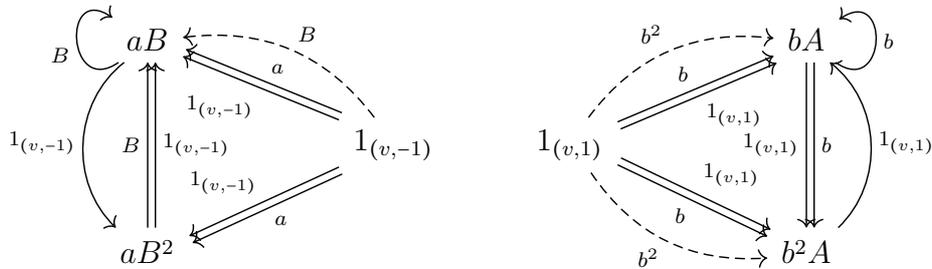

Note that in Figure \ref{extbrquiverex} the weak half bridge $\uu:=(1_{(v,-1)} \xrightarrow{B} aB)$ is not a half bridge because $\uu':=(1_{(v,-1)}\xrightarrow{1_{(v,-1)}}aB)$ satisfies $\uu'\sqsubseteq\uu$, $\uu\nsqsubseteq\uu'$ and $\sBa(\uu)=\sBa(\uu')$. On the other hand, although $1_{(v,1)}\xrightarrow{1_{(v,1)}}bA$ is strictly lower than  $1_{(v,1)}\xrightarrow{b}b^2A$ with respect to $\sqsubseteq$, the latter is a half bridge since the values of $\sBa$ associated to them are different.
\end{example}

\subsection{Generating strings with paths in bridge quivers}\label{roottrees}
The goal of this section is to show that each string is `generated by', in a precise sense, at least one directed path in the extended bridge quiver.

A \emph{weak path} is a directed path in $\widetilde{\Q}^\mathrm{Ba}$ or $\overline{\Q}^\mathrm{Ba}$ that begins and ends at a vertex not in $Q_0^\mathrm{Ba}$. A \emph{path} is a weak path of the form $\mathcal P:=(1_{(v,i)}\xrightarrow{\uu_0}\bb_1\xrightarrow{\uu_1}\bb_2\xrightarrow{\uu_2}\hdots\xrightarrow{\uu_{n-1}}\bb_n\xrightarrow{\uu_n}1_{(v',i')})$ for some $n\geq 0$. If $n=0$ then $\uu_0$ is a (weak) zero bridge otherwise $\uu_0$ is a (weak) half bridge, $\uu_{n}$ is a reverse (weak) half bridge while the rest $\uu_{i}$ are non-trivial (weak) bridges. Say that a string $\xx\in H_l(v)$ is \emph{generated by} $\mathcal P$ if $\xx$ is the string obtained from reducing the word of the form $\uu_n\bb_n^{m_n}\uu_n\hdots\uu_1\bb_1^{m_1}\uu_0$ such that for all $1\leq j\leq n$ we have $m_j\geq-1$. Analogous definition could be obtained for strings generated by weak paths where powers of prime bands are replaced by $1_{(v,i)}$ wherever appropriate.

\begin{rmk}\label{weaktononweak}
We need $m_j=-1$ in order to express a weak bridge, a weak (reverse) half bridge, or a weak zero bridge as a string generated by a finite path consisting of their non-weak counterparts.
\end{rmk}

Here is an example where negative power of a band is necessary. 
\begin{example}
Consider the following quiver
\begin{center}
\begin{tikzcd}
v_3                 &                                     & v_2 \arrow[ld, "d"] \arrow[ll, "c"'] \\
v_4 \arrow[r, "a"'] & v_1 \arrow[r, "e"'] \arrow[lu, "b"] & v_5                                 
\end{tikzcd}
\end{center}
with $\rho=\{ba, ed\}$. There is a unique band $\bb=cDB$. Then the string $ea$ is generated by the path $1_{(v_4,i)}\xrightarrow{cDa}\bb\xrightarrow{eB}1_{(v_5,j)}$ for appropriate $i,j\in\{1,-1\}$ as $ea$ can be obtained by simplifying the word $eB\bb^{-1}cDa$.
\end{example}

In a non-domestic string algebra there could be multiple paths generating the same string.
\begin{example}
Consider the extended bridge quiver of $GP_{2,3}$ from Figure \ref{extbrquiverex}. Then the string $B$ is generated by the path $1_{(v,-1)}\xrightarrow{1_{(v,-1)}} aB\xrightarrow{B}aB\xrightarrow{1_{(v,-1)}}1_{(v,-1)}$ as well as the path $1_{(v,-1)}\xrightarrow{1_{(v,-1)}}aB^2\xrightarrow{B}aB\xrightarrow{1_{(v,-1)}}1_{(v,-1)}$.
\end{example}

Let $\mathcal S(1_{(v,i)})$ denote the set of all strings generated by paths in the extended bridge quiver starting with a zero/half bridge from $1_{(v,i)}$. The following expected result states that we have indeed developed a language that accounts for all strings in a string algebra.
\begin{lemma}\label{generateallstrings}
For $v\in Q_0$, we have $\mathcal S(1_{(v,1)})=H_l(v)$ and $\mathcal S(1_{(v,-1)})=\{\xx^{-1}\mid\xx\in H_r(v)\}$.
\end{lemma}

The following result gives the inductive way to prove this lemma.
\begin{proposition}\label{weaktostrong}
Suppose $\ww$ is a band-free string and $\alpha$ is a syllable such that the composition $\alpha\ww$ is defined. Then there is a path in $\widetilde{\Q}_0^{\mathrm{Ba}}$ that generates $\alpha\ww$.
\end{proposition}

\begin{proof}
The string $\ww$ is band-free. If $\alpha\ww$ is band-free then $1_{(w',j')}\xrightarrow{\alpha\ww}1_{(w,j)}$ is the required path for appropriate $w,w'\in Q_0$ and $j,j'\in\{1,-1\}$.

On the other hand if $\alpha\ww$ is not band-free then there is a decomposition $\ww=\ww_1\ww_2$ where $\ww_1$ is of least length such that $\alpha\ww_1$ is a cyclic permutation of a band, say $\bb_\ww$. Due to the minimal length condition the band $\bb_\ww$ is a prime band.

If $\bb_\ww=\alpha\ww_1$ then the path $1_{(w',j')}\xrightarrow{\ww_2}\bb_\ww\xrightarrow{1_{(w,j)}}1_{(w,j)}$ generates $\alpha\ww$. Otherwise let $\bb_\ww=\uu'\uu,\alpha\ww_1=\uu\uu'$. Then the path $1_{(w',j')}\xrightarrow{\uu'\ww_2}\bb_\ww\xrightarrow{\uu}1_{(w,j)}$ generates $\alpha\ww$.
\hfill\end{proof}

\begin{proof} (of Lemma \ref{generateallstrings}) We only prove the first equality. By definition $\mathcal S(1_{(v,i)})\subseteq H_l(v)$, so it remains to prove the reverse inclusion.

Let $\uu\in H_l(v)$. Let $\uu=\alpha_p\hdots\alpha_1$. Let $\uu_k=\alpha_k\hdots\alpha_1$ for $1\leq k\leq p$. We prove that there is a path in the extended weak bridge quiver that generates $\uu_p$ by induction on $p$.

For $p=1$, the string $\uu=\alpha_1$ is itself a zero bridge. So, for induction, we assume the result for $p=k$ and prove for $p=k+1$.

Suppose $\mathcal P_k:=(1_{(v,1)}\xrightarrow{\uu_0}\bb_1\xrightarrow{\uu_1}\bb_2\xrightarrow{\uu_2}\hdots\xrightarrow{\uu_{n-1}}\bb_n\xrightarrow{\uu_n}1_{(v',i')})$ is a path in the extended weak bridge quiver that generates $\uu_k$. Then the weak path obtained by adjoining the zero bridge $1_{(v',i')}\xrightarrow{\alpha_{k+1}}1_{(w,j)}$ to $\mathcal P_k$ generates $\uu_{k+1}$. Using Proposition  \ref{weaktostrong} we could replace the part $\bb_n\xrightarrow{\uu_n}1_{(v',i')}\xrightarrow{\alpha_{k+1}}1_{(w,j)}$ of a weak path by a part of a path from $\bb_n$ to $1_{(w,j)}$ so as to obtain a path $\mathcal P_{k+1}$ in the extended weak bridge quiver generating $\uu_{k+1}$, thus completing the induction step.

Now it remains to show that there is a path in the extended bridge quiver that generates $\uu$. This can be achieved by using Remark \ref{weaktononweak} on the weak zero/half/reverse half/bridges appearing in the path obtained above in $\widetilde{\Q}^\mathrm{Ba}$.
\hfill\end{proof}

\subsection{Meta-$\bigcup$-cyclic string algebras}\label{MetaUcyclic}
We want to restrict our attention to a subclass of string algebras defined using some graph-theoretic property of its bridge quiver.

A \emph{meta-band} in $\Lambda$ is a directed cycle of positive length in its bridge quiver. Note that a string algebra is non-domestic if and only if it contains a meta-band.
\begin{definition}
Say that $\Lambda$ is \emph{meta-$\bigcup$-cyclic} if each connected component of its bridge quiver contains at least two vertices and can be written as a union of meta-bands.  
\end{definition}
A meta-$\bigcup$-cyclic string algebra is clearly non-domestic and each arrow of its bridge quiver lies in a meta-band. Moreover, $\Lambda$ is meta-$\bigcup$-cyclic if and only if each connected component of its bridge quiver is strongly connected and has at least two vertices.

We say a string $\xx$ is \emph{extendable} if it is a substring of a cyclic permutation of a band. The main goal of this section is to give an equivalent characterization (Theorem \ref{purenondomeqdef}) of a meta-$\bigcup$-cyclic string algebra in terms of extensibility of its strings.

We begin by studying which bands occur in the almost periodic strings of the form $\brac1\la(\xx_0)$. Say that a finite string $\uu$ is an \emph{$\la$-string} if it is a mixed string and is a substring of  $\brac1\la(1_{(v,i)})$ for some vertex $v$ and $i\in\{1,-1\}$.

\begin{rmk}\label{ltermstring}
Given any vertex $v$, $i\in\{1,-1\}$, a syllable $\alpha$ such that $\alpha1_{(v,i)}$ is defined, and $n>0$ there exists at most one $\la$-string $\uu$ of length $n$ such that $\alpha$ is the first syllable of $\uu$. 

An $\la$-string can always be extended on the left to an $\la$-string. In particular for an $\la$-string $\uu$ if there exists a direct syllable $\beta$ such that $\beta\uu$ exists then $\beta\uu$ is an $\la$-string, otherwise $\alpha\uu$ is an $\la$-string where $\alpha$ is an inverse syllable such that the composition exists.
\end{rmk}

This remark is the key to prove the following result.
\begin{proposition}\label{evenprim}
Suppose $\brac1\la(\xx_0)=\prescript{\infty}{}{\bb}\uu\xx_0$ for some string $\xx_0$ and band $\bb$, then $\bb$ is prime.
\end{proposition}

\begin{proof}
We will prove this result by contradiction. Suppose $\bb$ is composite. Let a permutation $\bb'$ of $\bb$ be written as $\bb'=\bb_n\hdots\bb_2\bb_1$ where $\bb_i$ are cyclic permutations of bands. By extending $\uu$ if necessary we may assume that $\brac1\la(\xx_0)=\prescript{\infty}{}{\bb'}\uu'\xx_0$. Let $v=s(\bb_i)$ for each $1\leq i\leq n$. By our assumption each $\bb_i$ and all cyclic permutations of $\bb'$ are $\la$-strings. 

If there is $1\leq i<n$ such that the first syllables of $\bb_i,\bb_{i+1}$ are both direct or inverse then their first syllables are same since $\bb_{i+1}\bb_i$ and $\bb_i^2$ are both defined. Now for $\xx=1_{(v,j)}$ with appropriate $j$, we have that both $\bb_{i-1}\hdots\bb_1\bb_n\hdots\bb_i\xx$ and $\bb_i\hdots\bb_1\bb_n\hdots\bb_{i+1}\xx$ are defined. Hence the above remark implies that these two strings are equal which is a contradiction. Therefore $n$ is even and the first syllables of consecutive $\bb_i$ are not both direct or both inverse.

Suppose the first syllable of $\bb_n$ is direct while that of $\bb_1$ is inverse then, since $\bb_n^2$ exists, $\bb_1\bb_n$ cannot be an $\la$-string in view of the latter part of the above remark which is again a contradiction. The case when the first syllable of $\bb_1$ is direct while that of $\bb_n$ is inverse is dealt with in a similar manner. So we can conclude that $n=1$, and hence that $\bb$ is prime.
\hfill\end{proof}

The next result is the key step towards the goal of this section.
\begin{lemma}
Suppose $\Lambda$ is a torsion-free string algebra. Then for a given string $\uu$ there exist strings $\vv,\vv'$ and prime bands $\bb,\bb'$ such that the word $\bb\vv\uu\vv'\bb'$ is defined.
\end{lemma}
\begin{proof}
First assume that $|\uu|>0$. Suppose $\uu=\alpha_k\alpha_{k-1}\hdots \alpha_1$ with $k\geq 1$. Without loss of generality we may assume that $\alpha_1$ is an inverse syllable. We first show that the string $\uu$ can be extended to the right.

If there is a direct syllable $\beta$ such that $\uu\beta$ is defined then $\ra(\uu)$ exists. Since $\Lambda$ is torsion-free we have that $\brac 1\ra(\uu)$ is also defined, say it has the form $\uu\vv'\bb'^\infty$ for some string $\vv'$ and prime band $\bb'$ thanks to Proposition \ref{evenprim}.

If there is no direct syllable $\beta$ such that $\uu\beta$ is defined then $\rb(\alpha_k\alpha_{k-1}\hdots \alpha_2)=\uu$. Once again since $\Lambda$ is torsion-free we get that $\brac1\rb(\uu)$ is defined, say it has the form $\uu\vv'\bb'^\infty$ for some string $\vv'$ and prime band $\bb'$ thanks to Proposition \ref{evenprim}.

In either case apply the dual argument to the string $\uu\vv'\bb'$ to obtain a string $\vv$ and a prime band $\bb$ such that $\bb\vv\uu\vv'\bb'$ is a string.

When $\uu=1_{(v,i)}$, since $\Lambda$ is torsion-free, the total degree of the vertex $v\in Q_0$ is at least 2. Thus there exists a syllable $\alpha$ such that $\alpha\uu$ is defined, and thus the above argument could be applied to the string $\alpha$.
\hfill\end{proof}

\begin{theorem}\label{purenondomeqdef}
Suppose $\Lambda$ is a torsion-free string algebra such that every connected component of its bridge quiver has at least two vertices. Then it is meta-$\bigcup$-cyclic if and only if every string in $\Lambda$ is a substring of a band.
\end{theorem}
\begin{proof}
Suppose $\uu\in\STR\Lambda$. Using the above lemma there exist strings $\vv,\vv'$ and prime bands $\bb,\bb'$ such that $\bb\vv\uu\vv'\bb'$ is a string.

Since the connected component of the bridge quiver containing $\bb,\bb'$ is strongly connected there is a path from $\bb$ to $\bb'$, say $\mathcal P:=\bb\xrightarrow{\uu_1}\bb_2\xrightarrow{\uu_2}\hdots\bb_{m-1}\xrightarrow{\uu_m}\bb'$. Now the string of the form $\vv\uu\vv'\bb'\uu_m\bb_{m-1}\hdots\uu_2\bb_2\uu_1\bb^p$ is a band for a suitable $p\geq 1$.

For the converse suppose that there is a connected component of the bridge quiver of $\Lambda$ that is not strongly connected.

Then we have a bridge $\bb_1\xrightarrow{\uu}\bb_2$ such that there is no directed path from $\bb_2$ to $\bb_1$ in the bridge quiver. Consider the string $\bb_2\uu\bb_1$. If it can be extended to a band $\uu'\bb_2\uu\bb_1$, then using Lemma \ref{generateallstrings} the string $\uu'$ would be generated by a path from $\bb_2$ to $\bb_1$ in the bridge quiver which is a contradiction to our choice.
\hfill\end{proof}

The definition of a prime band also yields the following result.
\begin{corollary}
Each arrow of a torsion-free meta-$\bigcup$-cyclic string algebra lies on a prime band.
\end{corollary}

We close this section by showing that no ``complicated'' real term labels a graph map between finite dimensional string modules for a meta-$\bigcup$-cyclic string algebra.
\begin{theorem}\label{norealterm}
For a meta-$\bigcup$-cyclic string algebra $\Lambda$ and any $\vv\in\STR\Lambda$ there is no string $\uu$ such that the following equation holds
\begin{equation}\label{eq1}
    \brac1\lb(\uu)=\brac1\la(\vv).
\end{equation}
\end{theorem}
\begin{proof}
Suppose $\brac1\la(\vv)$ exists for some $\vv\in\STR\Lambda$. From Theorem \ref{evenprim} we know that $\brac1\la(\vv)$ is of the form $\prescript{\infty}{}{\bb}\vv'\vv$ for a string $\vv'$ and a prime band $\bb$.

Since $\bb$ starts with an inverse syllable we also have $\brac1\la(1_{(v,i)})=\prescript{\infty}{}{\bb}$, where $v=s(\bb)$ and $i=-\sigma(\bb)$. It is clearly seen that Equation \eqref{eq1} has a solution if and only if the following equation has a solution. 
\begin{equation}\label{eq2}
\brac1\lb(\uu)=\brac1\la(1_{(v,i)}).
\end{equation}

Note that if a prime band $\bb$ has a direct (resp. inverse) syllable as an exit then for any substring $\uu$ of $\bb$ there is $n>0$ such that $\la^n(\uu)$ (resp. $\lb^n(\uu)$) is not a substring of $\bb^2$.

The equation $\brac1\la(1_{(v,i)})= \prescript{\infty}{}{\bb}$ implies that $\bb$ does not have a direct exit syllable. But since $\Lambda$ is meta-$\bigcup$-cyclic $\bb$ must have an exit and it must be an inverse syllable. Since any solution $\uu$ of Equation \eqref{eq2} is a substring of a finite power of $\bb$, in view of the discussion in the previous paragraph this equation does not have a solution.
\hfill\end{proof}

The result above can also be proved for the pairs $(\la,\lb),(\ra,\rb)$ and $(\rb,\ra)$.

\section{Computation of ranks of graph maps}\label{comprank}

\subsection{Recursive systems}\label{recsystem}
We want to develop a language to talk about graph maps between string modules. Given $\vv \in \STR\Lambda$ such that $\la(\vv)$ is defined and is equal to $a_k\hdots a_2a_1B\vv$, we set $\la_i(\vv) := a_{k-i}\hdots a_2a_1B\vv$ for $0\leq i\leq k$. Note that $\la_0(\vv)=\la(\vv)$. The notations $\lb_i(\vv), \ra_i(\vv), \rb_i(\vv)$ are defined in a similar way.
\begin{rmk}\label{prodli}
Since any string $\uu$ starting with an inverse syllable can be written as $\vv_pB_p\hdots \vv_2B_2\vv_1B_1$, where, for $1 \leq j \leq p$, $B_j$ are inverse syllables and $\vv_j$ are direct strings possibly of zero length, $\uu$ can be written as $\la_{k_p}\la_{k_{p-1}}\hdots \la_{k_1}(1_{(v,i)})$ which is a short hand for $\la_{k_p}(\la_{k_{p-1}} \hdots (\la_{k_1}(1_{(v,i)}))\hdots)$, where $v = s(B_1)$ and $i = -\sigma(B_1)$.
\end{rmk}

In case $\Lambda$ is a torsion-free domestic string algebra and $\la_i(\vv)$ exists for some $\vv \in \STR \Lambda$, there exists a real term $\mu$ such that $\mu(\vv) = \la_i(\vv)$ \cite{SK}.

As in \S\ref{Hammock}, we can ask if the equation 
\begin{equation*}
    \brac 1{\lb_j}(\xx)=\brac 1{\la_i}(\vv)
\end{equation*}
has a solution $\xx$ for a given string $\vv$. As described there we also assign meaning to the expressions $\brac{\lb_j}{\la_i}$ and $\brac{\la_i}{\lb_j}$ by replacing $\la=\la_0$ by $\la_i$ and $\lb=\lb_0$ by $\lb_j$. In a similar manner, we can also assign meaning to the expressions $\brac{\tau}{\mu}$ and $\brac{\mu}{\tau}$ where $\tau$ is a finite composition of $\la_{i_p}$ and $\mu$ is a finite composition of $\lb_{j_q}$. The expressions of the form $\tau,\mu, \brac{\tau}{\mu}$ and $\brac{\mu}{\tau}$ as defined above will be called \emph{complex terms}. We further say that the expressions of the form $\tau$ and $\brac{\mu}{\tau}$ (resp. $\mu$ and $\brac{\tau}{\mu}$) are \emph{$\la$-terms} (resp. \emph{$\lb$-terms}).

\begin{rmk}
Note that for strings $\vv,\xx$ such that the composition $\vv\xx$ is defined and $\la_i(\vv)$ exists, we have that $\la_i(\vv\xx)=\la_i(\vv)\xx$ whenever either (and hence both) sides are defined. Thus we denote the canonical inclusion map $\M(\vv) \to\M(\la_i(\vv))$ by $\La_i(v,j)$ where $v=t(\vv)$ and $j=\epsilon(\vv)$.
\end{rmk}
As in Remark \ref{prodli} if $\uu$ is a string starting with an inverse syllable and $j=-\sigma(\uu)$ then the canonical inclusion map $\M(1_{(v,j)})\to\M(\uu)$ can be written as a composition $\La_{k_p}\La_{k_{p-1}}\hdots \La_{k_1}(v,j)$. We say that the complex term labeling this graph map is $\la_{k_p}\la_{k_{p-1}}\hdots \la_{k_1}(v,j)$, where we also emphasize on the pair $(v,j)$.
\begin{definition}
Let $f$ be a canonical inclusion map between two string modules of $\Lambda$ such that the complex term $\tau:=\tau(v,j)$ labeling $f$ is an $\la$-term. We say that $\tau$ is a \emph{recursive term} if there are $\la$-terms $\tau_1,\tau_2$ and an $\lb$-term $\mu$ such that
\begin{equation*}
\tau(1_{(v,j)})= \brac{\mu}{\tau_1\tau\tau_2}(1_{(v,j)});
\end{equation*}
the data of the terms $\tau,\tau_1,\tau_2,\mu$ satisfying the above equation is said to constitute a \emph{recursive system}.
\end{definition}
\begin{example}
Let $\Lambda=GP_{2,3}$. Consider the graph map $f:\M(1_{(v,-1)})\to\M(B)$ labeled by the term $\la_1(v,-1)$. Then $\la_1$ satisfies the recursive equation $$\brac {\lb_1\lb} {\la\la_1\la}(1_{(v,-1)}) = \la_1(1_{(v,-1)}).$$
\end{example}
Maps labeled by recursive terms belong to the stable radical of $\Lambda$ which is the r of the next lemma. We will use the phrase `rank of a term' to refer to the rank of the graph map that it labels when the graph map is clear from the context.
\begin{lemma}\label{recursiveinfty}
Let $f$ be a canonical inclusion map between two distinct string modules of $\Lambda$ such that the complex term $\tau:=\tau(v,j)$ labeling $f$ is a recursive $\la$-term. Then $\rk(f)=\infty$.
\end{lemma}
\begin{proof}
Since $\tau$ is recursive there exist $\la$-terms $\tau_1,\tau_2$ and an $\lb$-term $\mu$ such that
\begin{equation*}
\tau(1_{(v,j)})= \brac{\mu}{\tau_1\tau\tau_2}(1_{(v,j)}).
\end{equation*}
Clearly $f$ is non-invertible and hence $\rk(\tau)\geq 1$. Since $\rad{\nu}$ is an ideal for all ordinals $\nu$ we have $\rk(\tau_1\tau\tau_2) \geq \rk(\tau)$.

We also have 
\begin{equation}\label{terms}
   \brac {\mu}{\tau_1\tau\tau_2} = \brac {\mu}{\tau_1\tau\tau_2}\tau_1\tau\tau_2.
\end{equation}

If $\rk(\tau_1\tau\tau_2) < \infty$ then by Lemma \ref{rkarithmetic} we have $\rk(\tau) = \rk(\brac{\mu}{\tau_1\tau\tau_2}) > \rk (\tau_1\tau\tau_2)$, which is a contradiction. Thus $\rk(\tau_1\tau\tau_2) = \infty$. Combining the above two equations we also have $\rk(\tau) \geq \rk(\tau_1\tau\tau_2)$. Therefore $\rk(\tau) = \infty$.
\hfill\end{proof}

Similarly we can define when a term labeling canonical inclusions or canonical surjections of the remaining three types is recursive and prove results like the above for them.

\subsection{The case of meta-$\bigcup$-cyclic string algebras}\label{casemetaucyclic}
In this section we prove that graph maps satisfying certain hypotheses lie in the stable radical. The first result guarantees that the hypotheses are satisfied somewhere.

\begin{proposition}\label{existinfrank}
Let $\Lambda$ be a string algebra and $\bb$ be a vertex of a meta-band. Then there exists a substring $\uu$ of a cyclic permutation of $\bb$ that starts with an inverse syllable, and distinct syllables $\alpha,\beta$ such that the strings $\alpha\uu,\beta\uu$ are defined and are extendable.
\end{proposition}

\begin{proof}
Suppose $\bb$ is a vertex of the meta-band $\bb_1(=\bb)\xrightarrow{\uu_1}\bb_2\xrightarrow{\uu_2}\hdots\xrightarrow{\uu_{n-1}}\bb_n\xrightarrow{\uu_n}\bb_1$, where $n\geq1$. Let $\alpha$ be the exit syllable of $\uu_1$ and $\beta$ the syllable of $\bb$ such that $s(\alpha)=s(\beta)$. It is easy to find a substring $\uu$ of a cyclic permutation of $\bb$ starting with an inverse syllable such that $\alpha\uu$ and $\beta\uu$ are defined. Then $\uu_n\bb_n\hdots\uu_2\bb_2\uu_1\bb^p$, for $p=1,2$, are distinct bands which extend $\alpha\uu$ and $\beta\uu$ respectively.
\hfill\end{proof}

In the next result we explicitly construct a recursive system to identify some elements of the stable radical.
\begin{lemma}\label{extendinfty}
Let $\Lambda$ be a meta-$\bigcup$-cyclic string algebra. Let $\xx$ be a string such that its first syllable is inverse and $\xx1_{(v,j)}$ is defined. If there is a direct syllable $\alpha$ and an inverse syllable $\beta$ such that $\alpha\xx$ and $\beta\xx$ are defined and are extendable. Then the rank of the canonical inclusion $f:\M(1_{(v,j)})\to\M(\xx)$ is $\infty$.
\end{lemma}

\begin{proof}
Suppose $\tau:=\tau(v,j)$ is a complex $\la$-term labeling $f$.

Let $\bb,\bb'$ be the cyclic permutations of bands extending $\alpha\xx,\beta\xx$ respectively such that $\bb1_{(v,j)}$ and $\bb'1_{(v,j)}$ are defined. Clearly $\bb\neq\bb'$. The composition $\bb\bb'$ is always defined while $\bb'\bb$ may not be a string. Since $\Lambda$ is meta-$\bigcup$-cyclic, there exists a string $\vv$ generated by a path from $\bb$ to $\bb'$ such that $\bb''=\bb'\vv\bb$ is defined and is a cyclic permutation of a band. Then $\bb''\bb'$ and $\bb'\bb''$ are strings.

Using Remark \ref{ltermstring} we can find complex $\la$-terms $\tau_1,\tau_2$ and a complex $\lb$-term $\mu$ such that $$\tau_1(1_{(v,j)})=\bb'',\ \tau_2(\xx) = \bb',\ \mu(\xx) = \xx\bb'\bb''.$$
Hence we get
\begin{equation*}
    \brac1\mu(\tau(1_{(v,j)}))=\brac1{\tau_2\tau\tau_1}(1_{(v,j)})
\end{equation*}
If $\mu=\lb_{n_p}\hdots\lb_{n_1}$ and $\xx=\lb_{m_u}\hdots\lb_{m_1}(\yy)$ for some substring $\yy$ of $\xx$ obtained using Remark \ref{prodli}, then $p>u$. Hence $\xx=\tau(1_{(v,j)})$ is the fundamental solution of the above equation to yield the following recursive system.
\begin{equation*}
    \tau(1_{(v,j)})=\brac\mu{\tau_2\tau\tau_1}(1_{(v,j)})
\end{equation*}
Thus $\tau$ is a recursive $\la$-term and hence by Lemma \ref{recursiveinfty}, we conclude $\rk(f)=\infty$.
\hfill\end{proof}

The hypotheses of the above lemma are quite restrictive and are relaxed below.

\begin{theorem}\label{infintervalinfrk}
Let $\Lambda$ be a meta-$\bigcup$-cyclic string algebra. Let $k>0$ and $v \in Q_0$ be such that $\xx:=\la_k(1_{(v,1)})$ exists. Suppose $f : \M(1_{(v, 1)}) \to \M(\la_k(1_{(v,1)}))$ denotes the canonical inclusion and $I:=[1_{(v, 1)}, \xx]$ denotes the closed interval in $H_l(v)$. If $I$ has infinitely many elements then $\rk(f)=\infty$.
\end{theorem}
\begin{proof}
Every element in the interior of $I$ is of the form $\yy a \xx$ where $\yy \in \STR \Lambda$ and $a \in Q_1$. The interval $I$ has infinitely many elements if and only if there exists a prime band $\bb'$ and a string $\zz$ such that $\bb'\zz a \xx$ is a string in $I$, thanks to Proposition \ref{finbandfree}.

Let $\uu$ be the string obtained from Proposition \ref{existinfrank} that is a substring of a cyclic permutation of $\bb$. Let $\bb=\uu''\uu'$ be such that $\uu\uu'$ is defined. Further suppose $\uu1_{(v',i)}$ is defined.
Let $\vv$ be a string generated by a path from $\bb'$ to $\bb$ in the bridge quiver such that $\uu\uu'\bb\vv\bb'\zz a \xx$ is a string. The canonical inclusion map $g : \M(\uu'\bb\vv\bb'\zz a \xx)  \to \M(\uu\uu'\bb\vv\bb'\zz a \xx)$ is in the stable radical by Lemma \ref{extendinfty} since the intervals $[1_{(v',i)},\uu]$ and $[\uu'\bb\vv\bb'\zz a \xx,\uu\uu'\bb\vv\bb'\zz a \xx]$ in appropriate hammocks are isomorphic. Since $f$ lies in the ideal generated by $g$ we also have that $\rk(f) = \infty$.
\hfill\end{proof}

The above two results could be combined to yield an alternate proof of \cite[Theorem~3(iii)]{SchroerInfRadMod}.

Clearly if the interval $I$ has only finitely many elements then $\rk(f)<\omega$. Now we are ready to prove the main result for meta-$\bigcup$-cyclic string algebras.
\begin{theorem}\label{PNDmain}
The rank of a graph map between finite dimensional string modules for a meta-$\bigcup$-cyclic string algebra $\Lambda$ is either finite or equals $\infty$.
\end{theorem}
\begin{proof}
By Theorem \ref{infintervalinfrk} a map labeled by $l_k$ for $k>0$ is either in the stable radical or is of finite rank. Moreover Proposition \ref{existinfrank} together with Lemma \ref{extendinfty} guarantees the existence of at least one graph map in the stable radical.

Using Remark \ref{prodli} we know that a graph map $f$ between finite dimensional string modules of $\Lambda$ is a finite composition of maps labeled by $\la_i$, $\ra_i$, $\lb_i$ and $\rb_i$. If any of these maps lies in the stable radical then so does $f$ since the stable radical is an ideal. Otherwise $f$ is of finite rank being a finite composition of finite rank maps.
\hfill\end{proof}

\subsection{The case of meta-torsion-free string algebras}\label{metatorsionfree}
The proof of Theorem \ref{infintervalinfrk} motivates the following definition.
\begin{definition}
Say that $\Lambda$ is \emph{meta-torsion-free} if each connected component of its bridge quiver contains at least two vertices and every directed path in its bridge quiver with positive finite length can be extended to a directed path indexed by $\Z$.
\end{definition}

This class of non-domestic string algebras can be characterized in terms of the rank of graph maps between its finite dimensional string modules.
\begin{theorem}\label{EPNDmain}
The following are equivalent for a non-domestic string algebra $\Lambda$.
\begin{enumerate}
    \item The algebra $\Lambda$ is meta-torsion-free.
    \item The rank of any graph map between finite dimensional string modules for $\Lambda$ is either finite or $\infty$.
\end{enumerate}
\end{theorem}

\begin{proof}
(1)$\implies$(2): This is the analogue of Theorem \ref{PNDmain} for meta-torsion-free string algebras. After making suitable modifications to the proofs of Proposition \ref{existinfrank}, Lemma \ref{extendinfty} and Theorem \ref{infintervalinfrk} the proof remains essentially the same as that of Theorem \ref{PNDmain}.

(1)$\impliedby$(2) We prove the contrapositive of the statement. Suppose $\Lambda$ is not meta-torsion-free. Without loss of generality we may assume that there exists a prime band $\bb$ such that any directed path starting with $\bb$ is finite. Thus there is a prime band $\bb'$ such that there is no non-trivial bridge from $\bb'$.

Let $\alpha$ be an inverse syllable of $\bb'$ such that the canonical inclusion map $f:\M(1_{(v,i)})\to\M(\alpha)$, where $v=s(\alpha),i=-\sigma(\alpha)$, is labeled by $\la_k$ for some $k>0$ and $\la_{k-1}(1_{(v,i)})$ is a substring of a cyclic permutation of $\bb'$. Up to inversion of all values of $\sigma,\epsilon$ maps for $\Lambda$, we may assume that $i=1$. 

If we extend the definition of $<_l$ to include left $\N$-strings we obtain an extension of $H_l(v)$, say $\hat H_l(v)$. Consider the interval $I:=[1_{(v,1)},\alpha]$ in $\hat H_l(v)$. From the choice of $\bb'$ and $\alpha$ it can be easily seen that $I$ contains only one infinite string, say $\ww$.

Any factorization of the graph map $f$ through a string module corresponds to a choice of a finite string $\vv$ in the open interval $(1_{(v,1)},\alpha)$. Depending on whether $\ww$ lies in the interval $[1_{(v,1)},\vv]$ or $[\vv,\alpha]$, we have that the other one has only finitely many elements. 

If the graph map $f$ factors through a band module $\mathrm{B}(\bb,n,\lambda)$, then $\bb$ is necessarily $\bb'$. We also have that $1_{(v,1)}$ is an image substring of $^\infty\bb^\infty$. Dually either $1_{(v,1)}$ or $\alpha$ is a factor substring of $^\infty\bb^\infty$. Clearly this is impossible, and hence no factorisation of $f$ through a band module is possible.

This shows that $f\in\rad\omega\setminus\rad{\omega+1}$. Thus we have obtained a map, namely $f$, that is neither of finite rank nor in the stable radical.
\hfill\end{proof}

As a consequence we can conclude that any map between two string modules which does not lie in the same connected component of the AR quiver is of rank $\infty$. We further analyse graph maps whose source and/or target is a band module to complete our study of the stable radical in case of meta-torsion-free algebras.
\begin{corollary}\label{mtfrk}
For a meta-torsion-free string algebra $\Lambda$ we have  $\omega\leq\st(\Lambda)\leq\omega+2$.
\end{corollary}
\begin{proof}
Let $\Lambda$ be a meta-torsion-free string algebra. Since $\Lambda$ is non-domestic, we know that $\st(\Lambda)\geq\omega$. To complete the proof we will show that $\rad{\omega+2}\subseteq\rad{\st(\Lambda)}$.

A map $f$ in $\rad{\omega+2}$ can be written as the composition of at least three maps in $\rad\omega$, and thus its type is a word of length $4$ using the alphabet $\{S,B\}$.

If such a word has at least $2$ occurrences of the letter $S$ then, in view of Theorem \ref{EPNDmain}, a map of such type is in the stable radical. If the word has one occurrence of the letter $S$, then it also contains at least one occurrence of a subword of type $BB$. Since a graph map of type $BB$ whose rank is not finite must factor through a direct sum of string modules in view of Remark \ref{sbbbrank}, we observe that $f$ admits a factorization through a finite direct sum of maps of type SS whose rank is $\infty$. Similar argument holds if the word has no occurrence of $S$, as there would be two disjoint occurrences of subwords of type $BB$. Thus we can conclude that every map whose type is given by a word of length $4$ must lie in the stable radical.
\hfill\end{proof}

\begin{rmk}\label{sbrank}
If for any band $\bb$, any image substring $\vv$ of $^\infty\bb^\infty$, and any cyclic permutation $\bb'$ of $\bb$ we have $\brac1\la(\vv)\neq\prescript{\infty}{}{\bb'}\vv$ or $\brac1\ra(\vv)\neq\vv\bb'^\infty$, then a graph map of type SB corresponding to the band $\bb$ and whose associated string is $\vv$ lies in the stable radical in view of Theorem \ref{EPNDmain}.

In particular, in view of Proposition \ref{evenprim}, whenever $\bb$ is composite then such map always lies in the stable radical. Dual statements could be made about graph maps of type BS.
\end{rmk}

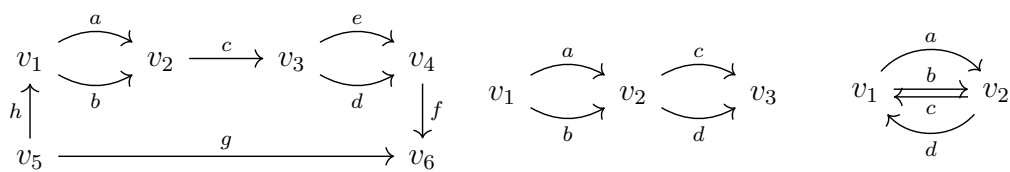
\begin{figure}[h]
\begin{minipage}[c]{0.4\textwidth}
\begin{tikzcd}
    v_1 \arrow[r, "a", bend left] \arrow[r, "b"', bend right] & v_2 \arrow[r, "c"] & v_3 \arrow[r, "e", bend left] \arrow[r, "d"', bend right] & v_4 \arrow[d, "f"] \\
    v_5 \arrow[u, "h"] \arrow[rrr, "g"]  &  &   & v_6           \end{tikzcd}
    \end{minipage}
    \begin{minipage}[c]{0.25\textwidth}
    \begin{tikzcd}
    v_1 \arrow[r, "a", bend left] \arrow[r, "b"', bend right] & v_2 \arrow[r, "c", bend left] \arrow[r, "d"', bend right] & v_3
    \end{tikzcd}
\end{minipage}
\hspace{20pt}
\begin{minipage}[c]{0.25\textwidth}
   \begin{tikzcd}
    v_1 \arrow[r, "b"] \arrow[r, "a", bend left=49] & v_2 \arrow[l, "c", shift left] \arrow[l, "d", bend left=49]
    \end{tikzcd}
\end{minipage}
\caption{Meta-torsion-free string algebras and their stable ranks \hspace{20pt}(L) $\Lambda:\rho=\{cb,dc,fd,bh\}, \st(\Lambda)=\omega$ ,(M) $\Lambda':\rho=\{cb,da\},\ \st(\Lambda')=\omega+1$, (R) $\Lambda'':\ \rho=\{cb,da,ad,bc,ac,bd\},\ \st(\Lambda'')=\omega+2$}
\label{Strkex}
\end{figure}

\begin{examples}\label{allranks}
Figure \ref{Strkex} shows examples of meta-torsion-free string algebras with stable ranks from the range obtained in Corollary \ref{mtfrk}.
\begin{enumerate}
    \item[$\Lambda:$] There are only three prime bands (up to inverse) and for each of them Remark \ref{sbrank} guarantees that the graph maps of type SB or BS, and hence those of type BB with non-finite rank, lie in the stable radical. Thus $\st(\Lambda)=\omega$. 
    
    \item[$\Lambda':$] There are only two prime bands (up to inverse), namely $\bb=cD$ and $\bb'=aB$. In view of Remark \ref{sbrank} the only graph maps of type SB which may not lie in the stable radical have underlying band $\bb$. Any image substring of $^\infty\bb^\infty$ has the form $\bb^n$ for some $n\geq0$. Note that $\bb^n$ cannot be a factor substring of any band. Moreover the canonical inclusion map $\M(\bb^n) \to \M(\bb^m)$ for $n \leq m$ is of finite rank. Hence a graph map from $\M(\bb^n)$ to $\mathrm B(\bb,n,\lambda)$ is of rank $\omega$.
    
    A dual argument for graph maps of type BS works for factor substrings of $^\infty\bb'^\infty$. Note that the maps of rank $\omega$ thus found, cannot be composed. Therefore, $\st(\Lambda') = \omega + 1$.
    
    \item[$\Lambda'':$] As in $\Lambda'$ there are only two prime bands (up to inverse), namely $\bb=cD,\bb'=aB$. As argued there, the only graph maps of rank $\omega$ are of the form $\M(\bb^n)\to\mathrm B(\bb,n,\lambda)$ and $\mathrm B(\bb',n',\lambda')\to\M((Ba)^m)$. For a suitable choice of $j$, the string $1_{(v_1,j)}$ is of the form $\bb^0$ as well as $(Ba)^0$, and this is the only such string. Therefore $\M(1_{(v_1,j)})$ must appear in any finite factorisation of the composition $\mathrm B(\bb',n',\lambda')\to\M(1_{(v_1,j)})\to\mathrm B(\bb,n,\lambda)$, which guarantees that the composition map is of rank $\omega+1$. Thus $\st(\Lambda'')=\omega+2$.
\end{enumerate}
\end{examples}

From the proof of Theorem \ref{EPNDmain} it is easy to obtain the following.
\begin{corollary}
If $\st(\Lambda) = \omega$ for a string algebra $\Lambda$ then $\Lambda$ is meta-torsion-free.
\end{corollary}

\subsection{Future directions}\label{fut}
In this paper we dealt with string algebras whose bridge quivers were essentially made of directed cycles. The other end of the spectrum, namely the class of domestic string algebras, is dealt with by the second and the third author in \cite{SK}. In a future work we hope to combine the ideas in these two papers to investigate the stable rank of an arbitrary special biserial algebra which is an algebra presented by $(\Q,\rho)$ where $\Q$ is a quiver and $\rho$ is a collection of relations that are not necessarily monomials. In particular it would be desirable to obtain the following generalization of \cite[Theorem~2]{SchroerInfRadMod}.
\begin{conj}
The stable rank, $\st(\Lambda)$, of a special biserial algebra $\Lambda$ with at least one band satisfies $\omega\leq\st(\Lambda)<\omega^2$.
\end{conj}

\end{document}